\documentclass[12pt]{amsart}

\usepackage{amsmath}
\usepackage{amssymb}
\usepackage[curve]{xypic}
\usepackage{caption}
\usepackage{xspace}
\usepackage{cite}
\usepackage{enumitem}
\usepackage{color}
\usepackage{url}

\makeatletter 
\def\@cite#1#2{{\m@th\upshape\bfseries%
[{#1\if@tempswa{\m@th\upshape\mdseries, #2}\fi}]}}
\makeatother 

\theoremstyle{plain}
\newtheorem{thm}{Theorem}[section]
\newtheorem{cor}[thm]{Corollary}

\newtheorem{lem}[thm]{Lemma}
\newtheorem{sublem}[thm]{Sublemma}
\theoremstyle{definition}
\newtheorem{defn}[thm]{Definition}

\newtheorem{prob}[thm]{Problem}

\theoremstyle{remark}
\newtheorem{rem}[thm]{Remark}

\numberwithin{equation}{subsection}
\renewcommand{\bold}[1]{\medskip \noindent {\bf #1 }\nopagebreak}

\newcommand{\nc}{\newcommand}
\newcommand{\rnc}{\renewcommand}

\newcommand{\bk}{{\mathbf{k}}}

\newcommand{\eb}[1]{\emph{\textbf{#1}}}

\nc\bA{\mathbb{A}}
\nc\bB{\mathbb{B}}
\nc\bC{\mathbb{C}}
\nc\bD{\mathbb{D}}
\nc\bE{\mathbb{E}}
\nc\bF{\mathbb{F}}
\nc\bG{\mathbb{G}}
\nc\bH{\mathbb{H}}
\nc\bI{\mathbb{I}}
\nc{\bJ}{\mathbb{J}} 
\nc\bK{\mathbb{K}}
\nc\bL{\mathbb{L}}
\nc\bM{\mathbb{M}}
\nc\bN{\mathbb{N}}
\nc\bO{\mathbb{O}}
\nc\bP{\mathbb{P}}
\nc\bQ{\mathbb{Q}}
\nc\bR{\mathbb{R}}
\nc\bS{\mathbb{S}}
\nc\bT{\mathbb{T}}
\nc\bU{\mathbb{U}}
\nc\bV{\mathbb{V}}
\nc\bW{\mathbb{W}}
\nc\bY{\mathbb{Y}}
\nc\bX{\mathbb{X}}
\nc\bZ{\mathbb{Z}}
\nc\cA{\mathcal{A}}
\nc\cB{\mathcal{B}}
\nc\cC{\mathcal{C}}
\rnc\cD{\mathcal{D}}
\nc\cE{\mathcal{E}}
\nc\cF{\mathcal{F}}
\nc\cG{\mathcal{G}}
\rnc\cH{\mathcal{H}}
\nc\cI{\mathcal{I}}
\nc{\cJ}{\mathcal{J}} 
\nc\cK{\mathcal{K}}
\rnc\cL{\mathcal{L}}
\nc\cM{\mathcal{M}}
\nc\cN{\mathcal{N}}
\nc\cO{\mathcal{O}}
\nc\cP{\mathcal{P}}
\nc\cQ{\mathcal{Q}}
\rnc\cR{\mathcal{R}}
\nc\cS{\mathcal{S}}
\nc\cT{\mathcal{T}}
\nc\cU{\mathcal{U}}
\nc\cV{\mathcal{V}}
\nc\cW{\mathcal{W}}
\nc\cY{\mathcal{Y}}
\nc\cX{\mathcal{X}}
\nc\cZ{\mathcal{Z}}

\nc{\dmo}{\DeclareMathOperator}
\dmo{\Tw}{Twist}
\dmo{\CP}{Pres}
\rnc{\Re}{\operatorname{Re}}
\rnc{\Im}{\operatorname{Im}}
\rnc{\span}{\operatorname{span}}
\dmo{\rank}{rank}
\dmo{\End}{End}
\dmo{\Jac}{Jac}
\dmo{\Id}{Id}
\dmo{\lcm}{lcm}

\nc{\Tm}{Teichm\"uller\xspace}

\begin{document}

\title[Cylinder deformations]{Cylinder deformations in orbit closures of translation surfaces}
%
\author[A.Wright]{Alex~Wright}
\address{Math\ Department\\University of Chicago\\
5734 South University Avenue\\
Chicago, IL 60637}
\email{alexmwright@gmail.com}
%

\begin{abstract}
Let $M$ be a translation surface. We show that certain deformations of $M$ supported on the set of all cylinders in a given direction remain in the $GL(2,\bR)$-orbit closure of $M$. Applications are given concerning complete periodicity, field of definition, and the number of of parallel cylinders which may be found on a translation surface in a given orbit closure.

The proof uses Eskin-Mirzakhani-Mohammadi's recent theorem on orbit closures of translation surfaces, as well as results of Minsky-Weiss and Smillie-Weiss on the dynamics of horocycle flow. 
\end{abstract}

\maketitle
\thispagestyle{empty}



\section{Introduction}\label{S:intro}
\bold{Context.} In this paper, we will demonstrate new and direct connections between the orbit closure of a translation surface and the flat geometry of cylinders on this translation surface.   

Several celebrated results already draw connections between the orbit closure of a translation surface and the cylinders on the translation surface.
\begin{itemize}
\item The Siegel-Veech formula establishes such a connection asymptotically in counting problems \cite{V4, EMa}.
\item The Veech Dichotomy shows in particular that for translation surfaces in closed orbits, a very strong form of complete periodicity holds \cite{V}.
\item Work of Calta and also very recent work of Lanneau and Nguyen show that in certain special orbit closures in low genus every translation surface is completely periodic \cite{Ca,LN}.
\end{itemize} 

Furthermore, the work of Smillie-Weiss \cite{SW2} gives that every \emph{horocycle flow} orbit closure contains a periodic translation surface (Theorem \ref{T:SW} below).

\bold{Main result.}The horocycle flow is defined as part of the $GL(2,\bR)$--action,
$$u_t=\left(\begin{array}{cc} 1&t\\0&1\end{array}\right)\subset GL(2,\bR).$$ 
We will also be interested in the vertical stretch, 
$$a_t=\left(\begin{array}{cc}1&0\\0& e^t\end{array}\right)\subset GL(2,\bR).$$ 

Given a collection $\cC$ of horizontal cylinders on a translation surface $M$, we define $u_t^\cC (M)$ to be the translation surface obtained by applying the horocycle flow to the cylinders in $\cC$ but not to the rest of $M$. This deformation of $M$ can be understood very concretely by expressing $M$ as a collection of polygons, including a rectangle for each cylinder in $\cC$, with parallel edge identifications. In this case, $u_t^\cC (M)$ is obtained by letting $u_t$ act linearly on the rectangles which give cylinders in $\cC$ but not on the remaining polygons, and then regluing. 

We call $u_t^\cC (M)$ the \emph{cylinder shear}. Similarly we define $a_t^\cC (M)$ by applying $a_t$ only to the cylinders in $\cC$, and we call $a_t^\cC (M)$ the \emph{cylinder stretch}. Both cylinder shear and stretch depend on the choice of a set $\cC$ of horizontal cylinders. Our main result is

\begin{thm}[The Cylinder Deformation Theorem]\label{T:main}
Let $M$ be a translation surface, and let $\cC$ be the collection of all horizontal cylinders on $M$. Then for all $s, t\in \bR$, the surface $a_s^\cC(u_t^\cC (M))$ remains in the $GL(2,\bR)$--orbit closure of $M$.
\end{thm}

We will in fact show that these \emph{cylinder deformations} remain in any affine invariant submanifold containing $M$. The key tool is the result of Smillie-Weiss \cite{SW2} mentioned above. Then we appeal to Eskin-Mirzakhani-Mohammadi's recent theorem that the $GL(2,\bR)$--orbit closure of every translation surface is an \emph{affine invariant submanifold} \cite{EM, EMM}, concluding the proof of Theorem \ref{T:main}.

The idea of the Cylinder Deformation Theorem was an outgrowth of conversations with Alex Eskin on the conjecture that affine invariant submanifolds are quasiprojective varieties. This conjecture restricts the types of linear equations on periods coordinates which may define an affine invariant submanifold in a way consistent with the Cylinder Deformation Theorem. We are profoundly grateful and deeply indebted to both Alex Eskin and Martin M\"oller for many helpful conversations on this conjecture. 

\begin{rem}
In genus 2, it is possible to recover Theorem \ref{T:main} from McMullen's classification of orbit closures \cite{Mc5}.
\end{rem}

\begin{rem}
Other interesting cylinder deformations can often be shown to remain in the orbit closure by considering the orbit closure of $u_t^\cC$ (see Corollary \ref{C:closure}). However, if $M$ has several horizontal cylinders, it is not always true that shearing or stretching just one cylinder, or shearing or stretching the different cylinders different amounts, will stay in the orbit closure. For example if $M$ is a lattice surface with more than one horizontal cylinder, then shearing just one horizontal cylinder will not remain in the orbit closure of $M$ (which is equal to the orbit of $M$ in this case). 
\end{rem}

\begin{rem}
There is no known algorithm to find all the cylinders in a given direction on a translation surface. Nonetheless, Alex Eskin has pointed out that the Cylinder Deformation Theorem might allow rigorous computer aided proofs that many translation surface have large orbit closures. Such a computer program might look for directions where all the cylinders can be determined (for example directions in which the translation surface is composed of cylinders and tori), and then use the Cylinder Deformation Theorem to produce vectors in the tangent space. 
\end{rem}

\bold{Complete periodicity.} Over any affine invariant submanifold $\cM$, there is a natural bundle $H^1_{rel}$ whose fiber over a translation surface $M$ is $H^1(M, \Sigma; \bC)$ (the cohomology of $M$ \emph{rel} $\Sigma$), where $\Sigma$ is the set of singularities of $M$. The tangent bundle $T(\cM)$ of $\cM$ is naturally a flat subbundle of $H^1_{rel}$. There is a natural map $p: H^1(M, \Sigma; \bC)\to H^1(M, \bC)$. 

A translation surface which is the union of cylinders in some direction is said to be \emph{periodic} in that direction.  A translation surface is said to be \emph{completely periodic} if whenever there is a cylinder in some direction, then that direction is periodic. 
\begin{thm}\label{T:CP}
If $\dim_\bC p(T(\cM)) = 2$, then every translation surface in $\cM$ is completely periodic.
\end{thm}

\begin{cor}\label{C:prym}
All translation surfaces in all the Prym eigenform loci are completely periodic. 
\end{cor}

Barak Weiss has pointed out to us that our proof of Theorem \ref{T:CP} and hence also Corollary \ref{C:prym} logically does not depend on Eskin-Mirzakhani-Mohammadi's recent work. In the case of Prym eigenform loci of dimension 5, Corollary \ref{C:prym} is due to Lanneau and Nguyen \cite{LN}; in the genus 2 case it is due to Calta \cite{Ca}; and for closed orbits it is part of the Veech dichotomy \cite{V}.

Lanneau has communicated to the author that there is a completely periodic translation surface $M$ in $\cH(4)^{hyp}$ which is not a lattice surface. Since this $M$ is not a lattice surface, its orbit closure $\cM$ satisfies $\dim_\bC T(\cM) > 2$. In the minimal stratum the map $p$ is an isomorphism (there are no relative periods), so for the example in question $\dim_\bC p(T(\cM))= \dim_\bC T(\cM)>2$.

Lanneau's example shows that there are completely periodic translation surfaces whose orbit closures have $\dim_\bC p(T(\cM)) > 2$. Thus the strongest converse to Theorem \ref{T:CP} which may be hoped for is 

\begin{thm}\label{T:CPconv}
If $\dim_\bC p(T(\cM)) > 2$, then there exist translation surfaces in $\cM$ which are not completely periodic.
\end{thm}

There are many examples of orbit closures $\cM$ with $\dim_\bC p(T(\cM)) = 2$ coming from Teichm\"uller curves (closed $SL(2,\bR)$--orbits) and ramified covering constructions. The first examples not of this form were discovered in genus 2 independently by McMullen and, from a different perspective, Calta \cite{Mc, Ca}. McMullen later generalized his approach to provide examples in genus $3, 4$, and $5$ \cite{Mc2}, giving the Prym eigenform loci mentioned above. There are no other currently known affine invariant submanifolds $\cM$ with $\dim_\bC p(T(\cM)) = 2$.

It is notable that Theorem \ref{T:CP} holds in all genera, whereas the methods of Calta, Lanneau and Nguyen, and the work of McMullen on flux \cite{Mc6}, exploit low genus phenomena. (McMullen has pointed out to us that some of his methods do extend to higher genus, see for example \cite[Thm. 1.1]{McCas}.)

On the one hand, this may be an indication that no high genus affine invariant submanifolds $\cM$ with $\dim_\bC p(T(\cM)) = 2$ exist which do not arise from covering constructions. On the other hand, Mirzakhani conjectures (see \cite{W3}) that any orbit closure $\cM$ which does not arise from a covering construction must satisfy $\dim_\bC p(T(\cM)) = 2$, and so if any genuinely new orbit closures are found we expect them to have this property.

\bold{Field of definition.} The field of definition of an affine invariant submanifold $\cM$ is the smallest subfield of $\bR$ such that $\cM$ can be defined in local period coordinates by linear equations with coefficients in this field. All linear equations in this paper will be assumed to be homogeneous (that is, have zero constant term). The field of definition of $\cM$ is denoted $\bk(\cM)$.

\begin{thm}[Wright, Thm. 1.1 in \cite{W3}]
The field of definition of any affine invariant submanifold $\cM$ is a number field of degree at most the genus. Moreover, it can be computed as the intersection of the holonomy fields of translation surfaces in $\cM$. 
\end{thm}

Here we present a different way of controlling the field of definition, in terms of only part of the flat geometry of a single translation surface in $\cM$.

\begin{thm}\label{T:kM1}
Suppose that $M$ is a translation surface in an affine invariant submanifold $\cM$, and suppose $M$ has at least one horizontal cylinder. ($M$ need not be horizontally periodic.) Let  $\{c_1, \ldots, c_n\}$ be the set of circumferences of horizontal cylinders on $M$. Then if $n>1$, $\bk(\cM)\subset \bQ[c_2/c_1, \ldots, c_n/c_1]$, and if $n=1$, $\bk(\cM)=\bQ$. 
\end{thm}

We wish to emphasize the $n=1$ case: If a translation surface has a direction with exactly one cylinder, then the field of definition of its orbit closure is $\bQ$.

Theorem \ref{T:kM1} is a special case of Theorem \ref{T:kM2}, which in particular additionally asserts the equality $\bk(\cM)= \bQ[c_2/c_1, \ldots, c_n/c_1]$ if the horizontal cylinders of $M$ remain parallel on nearby translation surfaces in $\cM$.

\bold{Finding many cylinders.} Let $T^M(\cM)$ denote the tangent space to $\cM$ at the point $M\in \cM$, and denote by $T^M(\cM)^*$ the space of linear functionals on $T^M(\cM)$. Any absolute or relative homology class $\alpha$ on $M$ naturally defines a linear functional $\alpha^*\in T^M(\cM)^*$, since $T^M(\cM)\subset H^1(M,\Sigma; \bC)= H_1(M,\Sigma; \bC)^* $. Continuing our effort to understand the connection between orbit closures and flat geometry, we prove

\begin{thm}\label{T:manyC}
Let $\cM$ be an affine invariant submanifold and set $$k=\frac12\dim_\bC p(T(\cM)).$$ Then there is some horizontally periodic translation surface $M\in\cM$ whose horizontal core curves span a subspace of $T^M(\cM)^*$ of dimension $k$. No set of core curves of parallel cylinders on a translation surface $M\in\cM$ may span a subspace of $T^M(\cM)^*$ of dimension greater than $k$. 

In particular, there is a horizontally periodic translation surface in $\cM$ with at least $k$ horizontal cylinders.  
\end{thm} 

The implied statement that $\dim_\bC p(T(\cM))$ is even is given by Avila-Eskin-M\"oller's result that $p(T(\cM))$ is symplectic \cite{AEM}. This is a key tool in the proof of Theorem \ref{T:manyC}. 

Theorem \ref{T:manyC} motivates 

\begin{defn}
The \emph{cylinder rank} of an affine invariant submanifold $\cM$ is $\frac12 \dim_\bC p(T(\cM))$. 
\end{defn}

We believe cylinder rank may be the most important numerical invariant of an affine invariant submanifold. Theorem \ref{T:CP} shows that complete periodicity prevails in rank 1 but not higher rank. Mirzakhani has conjectured that higher rank affine invariant submanifolds are arithmetic, i.e., have field of definition equal to $\bQ$. Theorems \ref{T:kM1} and \ref{T:manyC} were partially motivated by this conjecture. 

\bold{Background.} For an introduction to translation surfaces, see the surveys \cite{MT, Z}. 

There are many examples of affine invariant submanifolds which arise from ramified covering constructions. Examples not coming from covering constructions however are rare. Besides the examples mentioned above (following Theorem \ref{T:CPconv}), the only additional currently known examples are closed orbits: the Veech-Ward-Bouw-M\"oller curves \cite{V, W, BM, W2}, as well as two sporadic examples due to Vorobets and Kenyon-Smillie \cite{HS, KS} (see \cite{W2} for a more detailed summary). 

\bold{Organization.} Section \ref{S:periods} gives some definitions and describes the effect of cylinder shear in period coordinates. Section \ref{S:orbitclosure} describes the orbit closure of a cylinder shear. The key tool in the proof of Theorem \ref{T:main} is introduced in Section \ref{S:RD}, and the proof is completed in Section \ref{S:proof}. Sections \ref{S:CP}, \ref{S:field}, \ref{S:rank} and \ref{S:conv} prove Theorems \ref{T:CP}, \ref{T:kM1}, \ref{T:manyC}, and \ref{T:CPconv} respectively. Two open problems related to our work are listed in Section \ref{S:open}.

\bold{Acknowledgements.} The author thanks his thesis advisor Alex Eskin for inspiring this project, as well as Jon Chaika, Simion Filip, David Aulicino, Barak Weiss and John Smillie for enjoyable and helpful conversations on the results of this paper. The author is very grateful to Martin M\"oller and Maryam Mirzakhani for sharing their insight into orbit closures with the author, and to Erwan Lanneau for very helpful communications about complete periodicity in low genus. The author thanks Curtis McMullen for helpful comments on an earlier draft.

\section{Period coordinates}\label{S:periods}
Suppose $g\geq 1$ and let $\alpha$ be a partition of $2g-2$. The stratum $\cH(\alpha)$ is defined to be the set of $(X,\omega)$ where $X$ is a genus $g$ closed Riemann surface, and $\omega$ is a holomorphic 1-form on $X$ whose zeroes have multiplicities given by $\alpha$. For technical reasons, we prefer to work with a finite cover $\cH$ of $\cH(\alpha)$ which is a manifold instead of an orbifold.

Given a translation surface $(X,\omega)$, let $\Sigma\subset X$ denote the set of zeros of $\omega$. (We will also refer to $\Sigma$ as the set of singularities of $(X,\omega)$, since it is at these points where the flat metric is singular.) We consider the relative cohomology group $H_1(X,\Sigma; \bZ)$, and also the bundle $H^1_{rel}$ of relative cohomology, whose fiber over $(X,\omega)$ is  $H_1(X,\Sigma; \bC)$. For any neighborhood $\cU\subset \cH$ over which the bundle $H^1_{rel}$ is trivializable, we define the local period coordinate $\Phi:\cU\to H_1(X,\Sigma; \bC) \simeq \bC^m $ by 
$$\Phi(X,\omega)=[\omega].$$
Here $[\omega]$ is the relative cohomology class determined by $\omega$, and $m$ is the dimension of relative cohomology. 

\begin{defn} An \emph{affine invariant submanifold} of $\cH$ is a closed connected subset $\cM\subset \cH$ for which every point in $\cM$ has a neighborhood $\cU$ as above, satisfying $\Phi(\cU\cap \cM) = \Phi(\cU)\cap (V\otimes \bC)$, where $V$ is a subspace of $H^1(X,\Sigma;\bR)$. 
\end{defn}

More concretely, pick a basis $\gamma_1, \ldots, \gamma_m$ for $H_1(X,\Sigma; \bZ)$. The local period coordinates can be more explicitly written using the isomorphism $H_1(X,\Sigma; \bC) \simeq \bC^m$ defined by this choice of basis: 
$$\Phi(X,\omega)=\left( \int_{\gamma_i} \omega\right)_{i=1}^m.$$

An affine invariant submanifold is a submanifold of a stratum defined locally by homogenous linear equations with real coefficients on the period coordinates $\int_{\gamma_i} \omega$.

\begin{thm}[Eskin-Mirzakhani-Mohammadi, Thm. 2.1 in \cite{EMM}]
The $GL(2,\bR)$--orbit closure of any translation surface is an affine invariant submanifold. 
\end{thm}

In light of this, to establish Theorem \ref{T:main} it suffices to show that if $M=(X,\omega)$ is contained in an affine invariant submanifold $\cM$, then the cylinder shear and stretch remains in $\cM$, and this is what we will do. It will be necessary to understand the effect of cylinder shear and stretch in period coordinates. 

Given a collection $\cC$ of horizontal cylinders on a translation surface $M$, we define $\eta_\cC \in H^1(M,\Sigma; \bC)$ to be the derivative of the cylinder shear $u_t^\cC(M)$ is local period coordinates. That is, 
$$ \eta_\cC = \left.\frac{d}{dt}\right|_{t=0} \Phi(u_t^\cC(M)),$$
where $\Phi$ is a local period coordinate on a neighborhood of $M$. It is clear from the definition that $\eta_\cC$ is zero on any relative cycle not intersecting any of the cylinders in $\cC$, as well as on the core curves of the cylinders of $\cC$, but that $\eta_\cC$ is equal to the height of the cylinder on any relative cycle joining a zero on the bottom edge of a cylinder in $\cC$ to a zero on the top edge of this cylinder. Relative cycles of these three types span $H_1(X,\Sigma; \bZ)$, so this is a complete description of $\eta_\cC$. 

\begin{lem}
The deformation $u_t^\cC(M)$ is linear in local period coordinates:
$$\Phi(u_t^\cC(M)) = \Phi(M) + t \eta_\cC$$
for $t$ small enough. The cylinder stretch in local period coordinates is: 
$$\Phi(a_t^\cC(M)) = \Phi(M) + (e^t-1) i \eta_\cC.$$
\end{lem}

\begin{proof}
First set $(X_t, \omega_t)= u_t^\cC(M).$ It suffices to show that 
$$\int_\gamma \omega_t = \int_\gamma \omega + t \eta_\cC(\gamma)$$
for any set of $\gamma$ which spans $H_1(M, \Sigma; \bZ)$. However, this is clear for the three types of relative cycles described above. For $\gamma$ supported off the cylinders of $\cC$ or a core curve of a cylinder in $\cC$, it is clear that $\int_\gamma \omega_t$ is constant. For $\gamma$ a relative cycle joining a zero on the bottom edge of a cylinder in $\cC$ to a zero on the top edge of this cylinder, the imaginary part of $\int_\gamma \omega_t$ remains constant and the real part increases linearly in proportion to the height of the cylinder. 

The proof is extremely similar for the cylinder stretch $a_t^\cC(M)$.
\end{proof}

Note that our discussion also shows the following, which will be crucial in the sequel.

\begin{lem}\label{L:i}
Suppose that $\cC$ consists of cylinders $C_1, \ldots, C_r$, with core curves $\alpha_1, \ldots, \alpha_r$ and heights $h_1, \ldots, h_r\in \bR$. Let $I_{\alpha_i}\in H^1(M, \Sigma, \bZ)$ be the cohomology class defined to be zero on all relative homology classes which can be realized disjointly from the interior of $C_i$, and defined to be one on a relative cycle joining a zero on the bottom edge of $C_i$ to a zero on the top edge of this cylinder.
Then 
$$\eta_\cC = \sum_{i=1}^r h_i I_{\alpha_i},$$
\end{lem}

$I_{\alpha_i}$ should be thought of as ``intersection number with $\alpha_i$."

\begin{proof} 
Again this is verified on the spanning set for relative homology we have described. 
\end{proof}

\section{The orbit closure of a cylinder shear}\label{S:orbitclosure}

Part of the intuition for Theorem \ref{T:main} comes from the following. Given two sets of real numbers $S, S'\subset \bR$, we say they are \emph{independent} (over $\bQ$) if $\span_\bQ S \cap \span_\bQ S' = \{0\}$. 

\begin{lem}\label{L:hclosure}
Let $M$ be a horizontally periodic translation surface, and let $\cC_1$ be a collection of horizontal cylinders on $M$ such that the set of moduli of cylinders in $\cC_1$ is independent  from the set of moduli of the remaining horizontal cylinders on $M$. Then for all $t$, the cylinder shear $u_t^{\cC_1} (M)$ remains in the $GL(2,\bR)$--orbit closure of $M$.
\end{lem}

The proof is standard, but is included here for completeness. 

Given a collection of cylinders $\cC=\{C_1, \ldots, C_r\}$, we can also define the cylinder deformation $u_{t_1, \ldots, t_r}^\cC(M)$ which applies $u_{t_i}$ to $C_i$ for $i=1,\ldots, r$ and does nothing to the rest of the translation surface. Suppose that the height of $C_i$ is $h_i$, and the circumference is $c_i$. Then

\begin{lem}
$u_t^\cC(M)= u_{t_1, \ldots, t_r}^\cC(M)$, where $t_i= t \mod c_i/h_i$. 
\end{lem}

Here we use the notation $x\bmod y$ to denote the unique number $x'\in [0,y)$ which differs from $x$ by an integral multiple of $y$. The proof is left to the reader, and consists only of the observation that applying $u_{c_i/h_i}$ to a cylinder of circumference $c_i$ and height $h_i$ is equivalent to applying an element of the mapping class group: it is a full Dehn twist, and returns the original translation surface. 

\begin{lem}
Consider the flow $f_t(v)=v+t(1, \ldots, 1)$ on the torus 
$$\bT^r=[0, c_1/h_1)\times \cdots\times [0,c_r/h_r).$$ 
The orbit closure of $v$ is the set of all $v+(t_1/m_1,\ldots, t_r/m_r)$, where the $t_i$ satisfy all homogeneous linear relations with rational coefficients that are satisfied by the $m_i=h_i/c_i$. 
\end{lem}

The condition is that whenever $\sum_{i=1}^r q_i m_i=0$ with all $q_i\in \bQ$ then $\sum_{i=1}^r q_i t_i=0$ also. (We remind the reader that all linear equations in this paper are homogeneous unless otherwise specified.) 

\begin{proof}
Let $\bT_0=[0,1)^r$ be the standard torus, and define $\Psi:\bT\to \bT_0$ by 
$$\Psi(v_i)_{i=1}^r = (v_i m_i)_{i=1}^k.$$
Define the flow $g_t(w)= w+ t(m_1, \ldots, m_r)$ on $\bT_0$. We see that $\Phi(f_t(v))=g_t(\Psi(v))$. 
It is standard that the $g_t$--orbit closure of any $w\in \bT_0$ is the smallest  subtorus of $\bT_0$ containing $(m_1, \ldots, m_r)$, translated by $w$. That is, the orbit closure is the set of all $w+(t_1,\ldots, t_r)$ where the $t_i$ satisfy all rational homogeneous linear equations that the $m_i$ do. 

Moving back to $\bT$ using $\Psi^{-1}$, we see that the orbit closure of $v\in \bT$ is the set of all $v+ (t_1/m_1,\ldots, t_r/m_r)$ as claimed. 
\end{proof}

\begin{cor}\label{C:closure}
The $u_t^\cC$--orbit closure of $M$ is equal to the set of 
$$u_{t_1/m_1, \ldots, t_r/m_r}^\cC(M),$$ 
where the $t_i$ satisfy all homogeneous linear relations with rational coefficients that the $m_i=h_i/c_i$ do. In particular, if $\cC$ is the disjoint union of two collections of cylinders $\cC_1$ and $\cC_2$ whose sets of moduli are independent, then $u_t^{\cC_j}(M)$ is in the $u_t^\cC$--orbit closure of $M$, for all $t$ and $j=1, 2$. 
\end{cor}

\begin{proof}
The first statement follows directly from the previous two lemmas. For the second statement, assume that $\cC_1=\{C_1, \ldots, C_p\}$, and set $t_i=m_i$ for $i\leq p$ and $t_j=0$ otherwise. This choice of $t_i$ satisfies all the rational linear equations which the $m_i$ do, so by the first statement $u_t^{\cC_1}(M)=u_{tt_1/m_1, \ldots, tt_r/m_r}(M)$ is in the $u_t^\cC$--orbit closure of $M$. This proves the second statement for $j=1$, and the $j=2$ case is identical. 
\end{proof}

\begin{proof}[\eb{Proof of Lemma \ref{L:hclosure}.}]
Set $\cC$ to be the collection of all horizontal cylinders on $M$, so $u_t=u_t^\cC$. Set $\cC_2=\cC\setminus \cC_1$. Then by the previous corollary, $u_t^{\cC_1}(M)$ is in the $u_t$--orbit closure of $M$. 
\end{proof}

\section{Real deformations}\label{S:RD}

Let $M$ be a translation surface, and $\cU$ be a neighborhood of $M$ on which local period coordinates $\Phi$ are defined. 

\begin{defn}
A \emph{real deformation} of $M$ is any $M'\in \cU$ such that $\Phi(M')-\Phi(M)\in H^1(M,\Sigma;\bR)$. All real deformations can be obtained as $\Phi^{-1}(\Phi(M)+\zeta)$, where $\zeta\in H^1(M,\Sigma;\bR)$ is sufficiently small.
\end{defn}

\begin{rem}
Fix a translation surface $M$. There is a small simply connected neighborhood of $M$ in the stratum, in which every translation surface is \emph{marked} by a homeomorphism to $M$ \emph{rel} singularities which is unique up to isotopy \emph{rel} singularities. Any curve on $M$ which is either closed or joins a pair of singularities can thus be parallel transported to translation surfaces $M'$ in the neighborhood, and the result is unique up to isotopy \emph{rel} singularities. 

Now fix a finite set of saddle connections on $M$. By decreasing the size of the neighborhood of $M$, we may assume that on each $M'$ in the neighborhood the isotopy class of curve representing each saddle connection in the finite set is still represented by a saddle connection. The same discussion applies with flat regular geodesics on $M$ (i.e., core curves of cylinders).  

In this way we may make precise statements such as ``a cylinder on $M$ persists at nearby translation surfaces," and we may speak of this cylinder on translation surfaces sufficiently close to $M$. However, having now indicated the relevant standard techniques, we will omit technical details about markings in the remainder of this article. Whenever we speak of deformations of a translation surface $M$, or surfaces ``nearby" $M$, it is implicit that there exists a small neighborhood of $M$ on which the statements are true.  
\end{rem}

\begin{lem}\label{L:persist} 
The horizontal cylinders on $M$ persist under any real deformation, and maintain constant height. 
\end{lem}

Since horizontal cylinders persist, given a collection $\cC$ of horizontal cylinders on $M$, we will also speak of this collection of cylinders on any real deformation of $M$. 

\begin{proof}
Consider a path of real deformations $\Phi^{-1}(\Phi(M)+t\zeta)$, where $t$ is sufficiently small. Along this path, the imaginary part of the holonomy of every saddle connection does not change. Any cylinder on $M$ is bounded by a collection of horizontal saddle connections, and since these saddle connections remain horizontal the cylinder persists. Furthermore, since the height of a cylinder is the imaginary part of a saddle connection joining the lower and upper boundaries, this also remains constant. 
\end{proof}

Recall that $T^M(\cM)$ denotes the tangent space to $\cM$ at the point $M\in \cM$, and that $T^M(\cM)$ is naturally a subspace of $H^1(M,\Sigma; \bC)$, where $\Sigma\subset M$ is the set of singularities of $M$. 

\begin{defn}
Suppose $M\in \cM$. Two relative homology classes $\alpha,\beta\in H_1(M,\Sigma; \bZ)$ are called \emph{$\cM$--collinear}  if they have collinear images in $T^M(\cM)^*$. (Two nonzero vectors in a vector space are called collinear if they are scalar multiples.) 
\end{defn}

For example, if $\alpha$ and $2\beta$ are equal in $T^M(\cM)^*$ then $\alpha$ and $\beta$ are $\cM$--collinear, and this means exactly that for any translation surface in $\cM$ near $M$ the holonomy along $\beta$ is equal to twice the holonomy along $\alpha$. In other words, this equality must be one of the linear equations defining $\cM$ in local period coordinates at $M$.  

\begin{rem}
Relative homology classes $\alpha,\beta\in H_1(M,\Sigma; \bZ)$ also define functionals $\alpha^*, \beta^*$ on the real part of the tangent space $T^M_\bR(\cM)=T^M(\cM)\cap H^1(M, \Sigma; \bR)$. Since $T^M(\cM)$ is the complexification of $T^M_\bR(\cM)$, we see that $\alpha$ and $\beta$ are collinear in $T^M_\bR(\cM)^*$ if and only if they are collinear in $T^M(\cM)^*$. 
\end{rem}

\begin{defn}
Two saddle connections on $M\in\cM$ are \emph{$\cM$--parallel} if they are parallel at $M$ and at every nearby $M'\in \cM$. Two cylinders on $M\in\cM$ are \emph{$\cM$--parallel} if they are parallel at $M$ and at every nearby $M'\in \cM$.
\end{defn}

\begin{lem}
Two saddle connections are $\cM$--parallel if and only if their relative homology classes are $\cM$--collinear. Two cylinders are $\cM$--parallel if and only if their core curves are  $\cM$--collinear.
\end{lem}

\begin{proof}
This follows directly from the fact that $\cM$ can be defined in local period coordinates by linear equations with \emph{real} coefficients. 
\end{proof}

Being $\cM$--collinear or $\cM$--parallel is an equivalence relation. 

\begin{lem}\label{L:key1}
Suppose $M\in \cM$ has horizontal cylinders $C_1, \ldots, C_n$ with moduli $m_1, \ldots, m_n$. ($M$ is not necessarily horizontally periodic.) Suppose that there is a relation $$\sum_{i\in S} q_i m_i=0$$ where $0\neq q_i\in \bQ$ for all $i\in S$. Suppose furthermore that this relation holds not only at $M$ but also at every real deformation of $M$ in $\cM$, and that $S$ is minimal with this property.

Then the cylinders $C_i, i\in S$ are $\cM$--parallel. 
\end{lem}

The proof uses the following elementary lemma.

\begin{lem}\label{L:Findep}
Suppose that $V$ is a finite dimensional real vector space, and $F\subset V^*$ is a collection of linear functionals on $V$, no two of which are collinear. Then the collection of functions  
$\frac{1}{w}$ for $w\in F$
are linearly independent over $\bR$. This remains true when the functions are restricted to any nonempty open set in $V$.
\end{lem}

\begin{proof}
Consider a linear combination $g=\sum q_i w_i^{-1}$ with $q_i\in \bR$ and $w_i\in F$ distinct. Suppose some $q_i$ is non-zero. Pick 
$$v\in \ker(w_i)\setminus \bigcup_{j\neq i} \ker(w_j).$$
Note that $g$ has a pole at $v$, and so in particular $g$ is not zero. 

It follows that if $\sum q_i w_i^{-1}=0$, then all $q_i=0$. 

The restriction map from $\span w_i^{-1}$, restricting from $V$ to an open subset of $V$, is an injection. So the statement remains true when the functions are restricted to any nonempty open set in $V$.
\end{proof}

\begin{proof}[\textbf{\emph{Proof of Lemma \ref{L:key1}.}}]
Consider a neighborhood $\cV$ of $\Re\Phi(M)$ in $$V= T_\bR^M(\cM):= H^1(M,\Sigma;\bR)\cap T^M(\cM),$$ small enough that for all $v\in \cV$, $\Phi^{-1}(v+\Im\Phi(M))$ is a valid real deformation of $M$.  

For each $i$ we define a function $m_i$ on $\cV$, by setting $m_i(v)$ to be the modulus of the cylinder $C_i$ on $\Phi^{-1} (v+\Im\Phi(M)) $. Over all real deformations, the height $h_i$ of $C_i$ is constant, but the circumference varies according to the formula 
$$c_i= \alpha_i(v+\Im\Phi(M))= \alpha_i(v),$$
where $\alpha_i\in H_1(M,\bZ)$ is the core curve of the cylinder $C_i$. Hence we have 
$$m_i(v) = \frac{h_i}{\alpha_i(v)}.$$
Now suppose there is some $S\subset \{1, \ldots, n\}$, and $0\neq q_i\in \bQ$ so that 
$$\sum_{i\in S} q_i\frac{h_i}{\alpha_i(v)}=0$$
 for all $v\in \cV$, and suppose $S$ is minimal with this property. Then by the previous lemma, for all $i,j\in S$ we get that 
$\alpha_i(v)$ and $\alpha_j(v)$ are collinear as functionals on $T_\bR^M(\cM)$. Hence $\alpha_i$ and $\alpha_j$ are collinear in $T_\bR^M(\cM)^*$, and hence also in $T^M(\cM)^*$.
\end{proof} 

\begin{lem}
Suppose $M\in \cM$ is horizontally periodic, and $\cC$ is an equivalence class of $\cM$--parallel horizontal cylinders. Then there is a real deformation $M'\in \cM$ of $M$ where the set of moduli of the cylinders in $\cC$ becomes independent of the set of moduli of the remaining horizontal cylinders. 
\end{lem}

\begin{proof}
Suppose $\cC=\{C_1, \ldots, C_r\}$, and the remaining cylinders are $C_{r+1}, \ldots, C_{n}$. Let $m_i$ be the modulus of the cylinder $C_i$; so $m_i$ will be a function of the real deformation. If it is not possible to make $\{m_1,\ldots, m_r\}$ independent of $\{m_{r+1},\ldots, m_n\}$ using a real deformation in $\cM$, then there is some rational relation 
$$\sum_{i=1}^r q_i m_i = \sum_{j=r+1}^n q_j m_j$$
where all $q_i\in \bQ$, and this relation holds all real deformations of $M$, and neither the right hand side nor the left hand side is identically zero. 

Assume that we are given such a relation where the number of nonzero $q_i, i=1,\ldots, n$ is minimized. Let $S\subset \{1,\ldots, n\}$ be the set of $i$ for which $q_i\neq 0$. It is not hard to see that this $S$ is minimal, in that there is no rational relation among the $m_i$ where $i$ runs over any proper subset of $S$. 

Therefore if such a relation exists, by Lemma \ref{L:key1} we see that some $C_i$ for $i\leq r$ is $\cM$--parallel to some $C_j$ for $j>r$, which is a contradiction. 
\end{proof}

\begin{lem}\label{L:key2}
Suppose $M\in \cM$ is horizontally periodic, and $\cC$ an equivalence class of $\cM$--parallel horizontal cylinders.  

Then $u_t^\cC(M)$ remains in $\cM$ for all $t$. 
\end{lem}

\begin{proof}
It suffices to show that $\eta_\cC \in T^M(\cM)$, where $\eta_\cC$ is described in Section \ref{S:periods}. 

By Lemma \ref{L:key1} there is a real deformation $M'\in \cM$ of $M$ where the set of moduli of the cylinders in $\cC$ become independent of the set of moduli of the remaining horizontal cylinders. It is easy to check that the definition of $\eta_\cC$ is the same on $M$ and $M'$, since the heights of cylinders have not changed. 

By Lemma \ref{L:hclosure}, $\eta_\cC$ is in the tangent space of $\cM$ at $M'$. 
\end{proof}

\begin{lem}\label{L:key3}
Suppose that $M\in \cM$ is periodic in the horizontal direction, and let $\cC$ be a collection of horizontal cylinders on $M$, such that there is some deformation in $\cM$ of $M$ where the horizontal cylinders of $\cC$ remain horizontal but all other horizontal cylinders of $M$ do not. 

Then $u_t^\cC(M)$ remains in $\cM$ for all $t$. 
\end{lem}

\begin{proof}
No cylinder in $\cC$ can be $\cM$--parallel to a horizontal cylinder not in $\cC$. Thus $\cC$ is the union of equivalence classes $\cC_1, \ldots, \cC_s$ of $\cM$--parallel cylinders. The previous lemma gives that $\eta_{\cC_i}$ is in the tangent space to $\cM$ for each $i$, so it follows that $\eta_\cC=\sum_i \eta_{\cC_i}$ is in the tangent space as well. 
\end{proof}

\section{Proof of the Cylinder Deformation Theorem}\label{S:proof}

Our remaining key tool is

\begin{thm}[Smillie-Weiss, Cor. 6 in \cite{SW2}]\label{T:SW}
Every $u_t$--orbit closure contains a horizontally periodic translation surface. 
\end{thm}

Theorem \ref{T:SW} relies on the quantitative recurrence of horocycle flow, and the existence of compact minimal sets for the horocycle flow, established by Minsky-Weiss \cite{MW}. 

\begin{proof}[\textbf{\emph{Proof of Theorem \ref{T:main}.}}]
There is some horizontally periodic translation surface $M'$ on which the horocycle flow orbit of $M$ accumulates. Each horizontal cylinder of $M$ is ``present" on $M'$, and it has the same height and circumference on both $M$ and $M'$. 

Let $\cC$ be the set of horizontal cylinders of $M$. We will also speak of $\cC$ as being a collection of cylinders on $M'$. 

Since the horocycle flow orbit of $M$ accumulates at $M'$, there are arbitrarily small deformations of $M'$ where the cylinders of $\cC$ stay horizontal, but the remaining horizontal cylinders of $M'$ do not. These deformations are exactly $u_t(M)$ for $t$ very large, so that $u_t(M)$ is very close to $M'$. The only horizontal cylinders on $u_t(M)$ are the ones coming from $M$. So $u_t(M)$ is a deformation of $M'$ (just something close to $M'$) where the cylinders in $C$ remain horizontal, but all the other horizontal cylinders on $M'$ are no longer horizontal on $u_t(M)$.

Hence by Lemma \ref{L:key3}, $\eta_\cC$ is in the tangent space to $T_{M'}(\cM)$ at $M'$. Hence $\eta_\cC$ is in the tangent space to $T^M(\cM)$ at $M$. We conclude that the cylinder deformations of $M$, which result from adding $\eta_\cC$ or $i\eta_\cC$ in period coordinates, remain in $M$. 
\end{proof}

\section{Complete periodicity}\label{S:CP}

We now prove Theorem \ref{T:CP}. 

\begin{lem}\label{L:gamma}
Let $M$ be a translation surface that is not horizontally periodic. Then there is some curve $\gamma$ on $M$ which is supported off the union of the horizontal cylinders on $M$ for which the holonomy of $\gamma$ is not real (that is, not horizontal). 
\end{lem}

\begin{proof}
Let $x$ be some generic point in $M$ not contained in (or on the boundary of) any horizontal cylinder. (By generic, we mean generic with respect to the Poincar\'e Recurrence Theorem applied to the horizontal straight line flow in $M$.)  Flow from $x$ in the horizontal direction for a long time, until returning very close to $x$. Close up this path to get a curve $\gamma$ with the required property. 
\end{proof}

\begin{lem}\label{L:notinspan}
Let $M=(X,\omega)$ be a translation surface, and let $\cC$ be the collection of all horizontal cylinders on $M$. Suppose that $\cC$ is nonempty, but that $M$ is not horizontally periodic. Then $$p(\eta_\cC)\notin \span_\bC (\Re(p([\omega])), \Im(p([\omega]))).$$
\end{lem}

\begin{proof}
Suppose to the contrary that $$p(\eta_\cC)=c_1 \Re(p([\omega]))+c_2 \Im(p([\omega])).$$ Since $\eta_\cC$ evaluated at any core curve of a cylinder is zero, we see that $c_1=0$. 

Now consider the curve $\gamma$ produced by Lemma \ref{L:gamma}. Note that $p(\eta_\cC)(\gamma)= 0$, but $\Im([\omega])(\gamma)\neq 0$, so we see that $c_2=0$ as well. 

However, it is not hard to see that $p(\eta_\cC)\neq 0$, because it pairs nontrivially with homology classes whose intersection number with each horizontal core curve is positive (Lemma \ref{L:i}). 
\end{proof}

\begin{proof}[\eb{Proof of Theorem \ref{T:CP}.}]
Suppose that $M=(X,\omega)$ is not completely periodic. Without loss of generality, rotate $M$ so that there is at least one cylinder in the horizontal direction, but $M$ is not horizontally periodic. Let $\cC$ be the collection of horizontal cylinders on $M$. Let $\cM$ be any affine invariant submanifold containing $M$. By Theorem \ref{T:main}, we know that $\eta_\cC\in T^M(\cM)$. By Lemma \ref{L:notinspan} we see that $p(T^M(\cM))$ contains at least three linearly independent vectors, namely $\Re(p([\omega])), \Im(p([\omega])), p(\eta_\cC)$. 

We have shown that if $\cM$ contains a translation surface which is not completely periodic, then $\dim_\bC p(T(\cM))>2$. 
\end{proof}

\section{Field of definition}\label{S:field}

In this section we prove Theorem \ref{T:kM1}, which will be a corollary of the following more precise result. 

\begin{thm}\label{T:kM2}
Suppose that $M\in \cM$ has an equivalence class $\cC$ of $\cM$--parallel cylinders  consisting of $r\geq 1$ cylinders with circumferences $c_1, \ldots, c_r$. Then the field of definition of $\cM$ is $\bk(\cM)= \bQ[c_2/c_1, \ldots, c_r/c_1]$ if $r>1$ and $\bk(\cM)= \bQ$ if $r=1$.
\end{thm}


First we prove the easier direction in Theorem \ref{T:kM2}. 

\begin{lem}
Let $V\subset \bC^m$ be a linear subspace. Let $e_i^*$ be the coordinate functions, $i=1, \ldots, m$. If the equation $e_i^*-c e_j^*=0$ holds on $V$ for some $i\neq j$ and $c\in \bC$, but the equation $e_i^*=0$ does not hold on $V$, then the field of definition of $V$ contains $\bQ[c]$. 
\end{lem}

In this context, the field of definition of $V$ is the smallest subfield of $\bC$ with the property that $V$ may be described by linear equations in the $e_i^*$ with coefficients in this field. 

\begin{proof}[\eb{Proof of Lemma.}]
The projection of $V$ to the $i,j$ coordinate plane is a line of slope $c$. The field of definition of this line is $\bQ[c]$. If $V$ can be defined by linear equations with coefficients in some field, then so can its projection to any coordinate hyperplane. So it follows that the field of definition of $V$ contains $\bQ[c]$. 
\end{proof}

\begin{proof}[\eb{Proof that $k(\cM)\supset \bQ[c_2/c_1, \ldots, c_r/c_1]$.}]
Let $\cC=\{C_1, \ldots, C_r\}$, and let the core curve of the cylinder $C_i$ be denoted $\gamma_i$. 

Pick a basis of $H_1(M, \Sigma;\bZ)$, where $\Sigma\subset M$ is the set of cone points of $M$, so that this basis includes as many core curves of cylinders in $\cC$ is possible. Without loss of generality, assume that $\gamma_1, \ldots, \gamma_p$ are part of the basis. 

Note that $\bQ[c_2/c_1, \ldots, c_r/c_1]= \bQ[c_2/c_1, \ldots, c_p/c_1]$. Indeed, $\gamma_i$ for $i>p$ is in the $\bZ$--span of $\gamma_1, \ldots, \gamma_p$, so $c_i$ for $i>p$ is an integral linear combination of $c_1, \ldots, c_p$. 

Assume the rest of the basis is $\alpha_1, \ldots, \alpha_w$. For each $\gamma \in H_1(M,\Sigma;\bZ)$ we set $$\gamma^*(X,\omega)=\int_\gamma \omega,$$ so that $$(\gamma_1^*, \ldots, \gamma_p^*, \alpha_1^*, \ldots, \alpha_w^*)$$ give local coordinates for $\cM$ near $M$. The field of definition of $\cM$ is the smallest subfield of $\bR$ so that $\cM$ is defined in these local coordinates by linear equations with coefficients in this subfield. 

The linear equations $\gamma_i^*-c_i/c_1 \gamma_1^*=0$ hold on $\cM$ in these local coordinates. It follows from the previous lemma that $k(\cM)\supset \bQ[c_2/c_1, \ldots, c_p/c_1]$.
\end{proof}

To proceed, we require the following technical lemma. 

\begin{lem}\label{L:eta}
Suppose that $M\in \cM$ is a translation surface, and suppose $\cC$ is an equivalence class of $\cM$--parallel horizontal cylinders of circumferences $c_1, \ldots, c_r$. There is a tangent vector $\eta\in T^M(\cM)$ which is defined over $\bQ[c_2/c_1, \ldots, c_r/c_1]$ and for which $p(\eta)\neq 0$. 
\end{lem}

When we say that $\eta$ is \emph{defined over} some field, we mean that the real and imaginary parts of the periods of $\eta$ over any integral relative homology class are contained in this field. 

\begin{proof}
Let the cylinders in $\cC$ be $C_1,\ldots, C_r$. Let the height, circumference and modulus of $C_i$ be $h_i$, $c_i$ and $m_i=h_i/c_i$ respectively. 
By Corollary \ref{C:closure}, if $t_1, \ldots, t_r$ satisfy all the rational linear relations which the $m_i$ do, then $u_{tt_1/m_1, \ldots, tt_r/m_r}^\cC(M)$ remains in $\cM$ for all $t\in \bR$. Differentiating this at $t=0$, we find that $\eta\in T^M(\cM)$, where
$$\eta = \sum_{i=1}^r \frac{t_i}{m_i} h_i I_{\alpha_i} = \sum_{i=1}^r t_ic_i I_{\alpha_i}.$$
Compare to Lemma \ref{L:i}.

The restrictions on the $t_i$ are a set of rational linear equations which has nonzero solutions. Therefore this system has a nonzero solution with all $t_i\in \bQ$. Using a scaling of this solution $t_i'=t_i/c_1$ gives the result. 
\end{proof}

We also require

\begin{thm}[The Simplicity Theorem, Thm. 5.1 in \cite{W3}]
The only proper flat subbundles of $T(\cM)$ are contained in $\ker(p)$. 
\end{thm}

\begin{proof}[\eb{Proof of Theorem \ref{T:kM2}.}]
It remains only to show that $\bk(\cM)$ is contained in $\bQ[c_2/c_1, \ldots, c_r/c_1]$. Let $\eta\in T^M(\cM)$ be given by Lemma \ref{L:eta}. Since $p(\eta)\neq 0$, the Simplicity Theorem gives that the orbit of $\eta$ under the monodromy representation of $H^1_{rel}$ spans $T^M(\cM)$. Since the monodromy of $H^1_{rel}$ is integral, this gives that $T^M(\cM)$ is spanned by vectors defined over $\bQ[c_2/c_1, \ldots, c_r/c_1]$. Hence $T^M(\cM)$ is defined over $\bQ[c_2/c_1, \ldots, c_r/c_1]$, that is, $\bk(\cM)\subset \bQ[c_2/c_1, \ldots, c_r/c_1]$. (Similar arguments are explained in more detail in \cite{W3}.) 
\end{proof}

\begin{proof}[\eb{Proof of Theorem \ref{T:kM1}.}]
Without loss of generality, assume that the horizontal cylinders $C_1, \ldots, C_r$ form an equivalence class of $\cM$--parallel cylinders, where $r\leq n$. By Theorem \ref{T:kM2}, 
$$\bk(\cM)= \bQ[c_2/c_1, \ldots, c_r/c_1]\subset \bQ[c_2/c_1, \ldots, c_n/c_1].$$
\end{proof}

\section{Finding many cylinders}\label{S:rank}

This section is devoted to proving Theorem \ref{T:manyC}. The proof will consider the \emph{twist space} and the \emph{cylinder preserving space} of a horizontally periodic translation surface in some affine invariant submanifold. 

\begin{defn}
Let $M\in \cM$ be horizontally periodic. The \emph{twist space} of $\cM$ at $M$ is the subspace $\Tw(M, \cM)$ of $T_\bR^M(\cM)$ of cohomology classes in $T_\bR^M(\cM)$ which are zero on all horizontal saddle connections.
\end{defn}

\begin{lem} 
Let $\cH$ be a stratum, and let $M\in \cH$ be horizontally periodic. Let $C_1, \ldots, C_n$ be the horizontal cylinders on $M$. Then the twist space of $\cH$ at $M$ is 
$$\Tw(M,\cM) = span_\bR(\eta_{C_i})_{i=1}^n.$$
\end{lem}

\begin{proof}
Classes in $\span_\bR(\eta_{C_i})_{i=1}^n$ are zero on horizontal saddle connections, and hence 
$$\Tw(M,\cM) \supset \span_\bR(\eta_{C_i})_{i=1}^n.$$

Pick one ``cross curve" for each horizontal cylinder--a saddle connection from a zero on the bottom of the cylinder to a zero on the top. These cross curves, together with the horizontal saddle connections, span relative homology.

Consider a class $\eta \in \Tw(M,\cM).$ By definition, it is zero on all horizontal saddle connections. Suppose that the value of $\eta$ on the cross curve for $C_i$ is $b_i$. Then we will see that 
$$\eta=\sum_{i=1}^n \frac{b_i}{h_i} \eta_{C_i},$$
where $h_i$ is the height of $C_i$. Indeed, both the left hand side and the right hand side are zero on horizontal saddle connections, and by Lemma \ref{L:i}, the right hand side takes value $b_i$ on any cross curve for $C_i$. Since the right hand side and the left hand side agree on a basis for relative homology, they are equal. 
\end{proof}

\begin{cor}\label{C:tw}
Let $M\in \cM$ be horizontally periodic, with horizontal cylinders $C_1, \ldots, C_n$. Then 
$$\Tw(M, \cM)= T_\bR^M(\cM) \cap \span_\bR(\eta_{C_i})_{i=1}^n.$$
\end{cor}

\begin{proof}
By definition, $\Tw(M, \cM)$ is the set of real cohomology classes on $M$ which are in $T_\bR^M(\cM)$ and which are zero on all horizontal saddle connections. By the previous lemma, the space of real cohomology classes on $M$ which are zero on all horizontal saddle connections is exactly $\span_\bR(\eta_{C_i})_{i=1}^n.$
\end{proof}

\begin{defn}
Let $M\in \cM$ be horizontally periodic. The \emph{cylinder preserving space} of $\cM$ at $M$ is the subspace $\CP(M, \cM)$ of $T_\bR^M$ of cohomology classes which are zero on the core curves of all horizontal cylinders.
\end{defn}

Note that $\Tw(M,\cM)\subset \CP(M,\cM)$. 

\begin{lem}\label{L:preserve}
Suppose $\eta\in \CP(M,\cM)\otimes \bC$ is in the complexification of the cylinder preserving space of $\cM$ at $M$, and consider the deformation $M_t$ defined by $\Phi(M_t)=\Phi(M)+t\eta$. For $t$ sufficiently small, the horizontal cylinders on $M$ persist on $M_t$.  
\end{lem}

$\CP(M,\cM)\otimes \bC$ is the subspace of $T^M(\cM)$ of complex cohomology classes which are zero on the core curves of all horizontal cylinders. 

\begin{proof}
Indeed, the core curve of each horizontal cylinder on $M$ continues to have real holonomy at $M_t$. 
\end{proof}

We now arrive at the key idea of this section, which will be used to prove Theorem \ref{T:manyC}. 

\begin{lem}\label{L:key}
Suppose $M\in \cM$ is horizontally periodic, and $$\Tw(M,\cM) \subsetneq \CP(M,\cM).$$ Then there is a horizontally periodic translation surface in $\cM$ with more horizontal cylinders than $M$. 
\end{lem}

\begin{sublem}
Suppose $M\in \cM$ is horizontally periodic, and 
$$\eta \in \CP(M,\cM) \setminus \Tw(M,\cM)$$
is sufficient small. Consider the deformation $M'$ of $M$ defined by $\Phi(M')=\Phi(M)+i\eta$. Then all of the horizontal cylinders from $M$ persist on $M'$, but their union is not all of $M'$. 
\end{sublem}

\begin{proof}[\eb{Proof of Sublemma.}]
Because $i\eta$ is in the complexification of the cylinder preserving space, Lemma \ref{L:preserve} gives that the horizontal cylinders from $M$ persist on $M'$, so it remains only to show that their union is not all of $M'$. 

Because $\eta$ is not in the twist space, it takes a nonzero value on some horizontal saddle connection. On $M'$, this saddle connection is no longer horizontal, but instead acquires some very small vertical part which may be assumed to be vastly smaller than the height of any horizontal cylinder on $M$. 

Thus it can be arranged for $M'$ to have a saddle connection whose vertical part is vastly smaller than the heights on $M'$ of the cylinders which have persisted from $M$. The union of the cylinders on this $M'$ which have persisted from $M$ is not all of $M'$; in particular, the union cannot contain the saddle connection under consideration. 
\end{proof}

\begin{proof}[\eb{Proof of Lemma \ref{L:key}.}]
By assumption we may find $\eta \in \CP(M,\cM) \setminus \Tw(M,\cM)$. The sublemma thus gives a deformation $M'\in \cM$ of $M$ where the horizontal cylinders of $M$ persist but do not cover the whole surface. By Theorem \ref{T:SW} we may find $M''$ in the horocycle flow orbit closure of $M'$ which is horizontally periodic. The horizontal cylinders on $M'$ persist on $M''$; in particular, for each horizontal cylinder on $M$ there is a corresponding horizontal cylinder on $M''$, but these horizontal cylinders do not cover $M''$. So $M''$ must have more horizontal cylinders than $M$.   
\end{proof}

To apply Lemma \ref{L:key} we will need to control the size of the twist space and the cylinder preserving space. Controlling the size of the cylinder preserving space is straightforward. 

\begin{lem}\label{L:codimd}
Suppose that $M\in \cM$ is horizontally periodic. Suppose that the core curves of the horizontal cylinders on $M$ span a subspace of $T^M(\cM)^*$ of dimension $d$. Then $\CP(M,\cM)$ has codimension $d$ in $T_\bR^M(\cM)$. 
\end{lem}

\begin{proof}
Suppose that the core curves of the horizontal cylinders on $M$ are $\alpha_1, \ldots, \alpha_n$. Then the subspace $\CP(M,\cM)$ of $T_\bR^M(\cM)$ is defined by the equations $\alpha_i^*=0, i=1, \ldots, n$.
\end{proof}

Controlling the size of the twist space is more difficult. For this we require 

\begin{thm}[Avila-Eskin-M\"oller, Thm. 1.4 in \cite{AEM}]\label{T:AEM}
$p(T(\cM))$ is symplectic. 
\end{thm}

\begin{lem}\label{L:iso}
Let $M \in \cM$ be horizontally periodic. Then $p(\Tw(M,\cM))$ is isotropic. 
\end{lem}

\begin{proof}
Let $C_1, \ldots, C_n$ be the horizontal cylinders on $M$. By Corollary \ref{C:tw}, it suffices to prove that 
$$\span_\bR(p(\eta_{C_i}))_{i=1}^n$$
is isotropic. Lemma \ref{L:i} gives that $\eta_{C_i}$ is, up to a scalar multiple, equal to intersection number with the core curve of $C_i$. Since the core curves of the $C_i$ are disjoint, the result follows. 
\end{proof}

\begin{cor}\label{C:cd}
Let $k=\frac12 \dim_\bC p(T(\cM))$. Then $\Tw(M,\cM)$ has codimension at least $k$ in $T_\bR^M(\cM)$.
\end{cor}

\begin{proof}
The maximal dimension of an isotropic subspace of $p(T(\cM))$ is $k$. Hence $p(\Tw(M,\cM))$ has codimension at least $k$ in $p(T_\bR^M(\cM))$. It follows that $\Tw(M,\cM)$ has codimension at least $k$ in $T_\bR^M(\cM)$.
\end{proof}

\begin{cor}\label{C:more}
Let $M\in \cM$ be horizontally periodic, and suppose $k= \frac12\dim_\bC p(T(\cM))$. Suppose that the core curves of the horizontal cylinders on $M$ span a subspace of $T^M(\cM)^*$ of dimension $d$. If $d<k$, then there is a translation surface $M'\in \cM$ with more horizontal cylinders than $M$. 
\end{cor}

\begin{proof}
If $d<k$, then by Lemma \ref{L:codimd} and Corollary \ref{C:cd}, the cylinder preserving space at $M$ must be strictly larger than the twist space at $M$. Hence the result follows by Lemma \ref{L:key}. 
\end{proof}

\begin{proof}[\eb{Proof of Theorem \ref{T:manyC}.}]
Pick $M\in \cM$ with the maximal number of parallel cylinders. Without loss of generality, we may assume that the cylinders are horizontal. In this case $M$ must be horizontally periodic, or else Theorem \ref{T:SW} would produce a horizontally periodic translation surface in the $h_t$--orbit closure, which  would necessarily contain more horizontal cylinders. 

By Corollary \ref{C:more}, the core curves of the horizontal cylinders on $M$ span a subspace of $p(T^M(\cM))^*$ of dimension at least $k$. In particular, there must be at least $k$ equivalence classes of $\cM$--parallel horizontal cylinders on $M$. 

This shows the first statement of Theorem \ref{T:manyC}. The second (converse) statement follows directly from Theorem \ref{T:AEM}: the core curves of parallel cylinders spans an isotropic subspace of $p(T^M(\cM))^*$, and any isotropic subspace has dimension at most $k$. Since the core curves of the cylinders are absolute homology classes, their span in $p(T^M(\cM))$ is isomorphic to their span in $T^M(\cM)$. 
\end{proof}


\section{Proof of Theorem \ref{T:CPconv}}\label{S:conv} 

From Theorem \ref{T:manyC}, we will need only

\begin{cor}\label{C:atleast2}
If $\frac12 \dim_\bC p(T(\cM))>1$ then there is a translation surface in $\cM$ with at least two equivalence classes of $\cM$-parallel horizontal cylinders.
\end{cor}

\begin{lem}\label{L:ds}
Suppose $\cC_1, \ldots, \cC_s$ are some equivalence classes of $\cM$-parallel horizontal cylinders on $M\in \cM$, whose core curves span a subspace of $T^M(\cM)^*$ of dimension $d$. (There might be additional equivalence classes of $\cM$-parallel horizontal cylinders on $M$.) The set of deformations of $M$ where these $s$ equivalence classes of $\cM$-parallel horizontal cylinders remain horizontal is locally a manifold of real codimension $d$. 
\end{lem}

\begin{proof}
Let $\alpha_i$ be the core curve of a cylinder in $\cC_i$. In local period coordinates, the set of deformations in question is defined by $\Im(\alpha_i^*)=0, i=1,\ldots, s.$
\end{proof}

\begin{cor}\label{C:notmany}
The set of translation surfaces in $\cM$ with at least two equivalence classes of $\cM$-parallel horizontal cylinders is the countable union of immersed submanifolds of real codimension at least two. 
\end{cor}

\begin{proof}
This follows immediately from Lemma \ref{L:ds}.
\end{proof}

\begin{proof}[\eb{Proof of Theorem \ref{T:CPconv}}]
By Corollary \ref{C:atleast2}, we may find a translation surface $M$ with at least two equivalence classes of $\cM$-parallel horizontal cylinders. Lemma \ref{L:ds} gives that there is a immersed submanifold of real codimension 1 where one of these equivalence classes of $\cM$-parallel cylinders remains horizontal. 

By Corollary \ref{C:notmany}, not all of these translation surfaces can have at least two equivalence classes of $\cM$-parallel horizontal cylinders. Therefore there is a deformation of $M$ with exactly one equivalence class of $\cM$-parallel horizontal cylinders. This equivalence class cannot cover all of the deformation, since it did not cover all of $M$. 

This deformation has at least one horizontal cylinder, but is not horizontally periodic. 
\end{proof}

\section{Two open problems}\label{S:open} 

The Cylinder Deformation Theorem leads naturally to two problems, both of which were already of interest before our work. 

\begin{prob}\label{p1} 
Find an algorithm which, given a translation surface, finds all the cylinders in any given direction. 
\end{prob}

We do not know if such an algorithm exists, however even solutions in special cases would be interesting. For a solution to be useful, the algorithm must be guaranteed to terminate. (It is easy to give algorithms which can find cylinders in a given direction on a translation surface, but it seems hard in general to certify that all the cylinders in a given direction have been found.)   

\begin{prob}\label{p2}
Let $M$ be a surface which is \emph{completely parabolic}, which by definition means $M$ is completely periodic and that any two parallel cylinders on $M$ have rationally related moduli. Must $M$ be a lattice surface? 
\end{prob}

Problem \ref{p2} is in fact a problem of Smillie and Weiss and appeared as Question 5 of \cite{SWprobs}. For completely parabolic surfaces $M$, the Cylinder Deformation Theorem yields nothing. Even considering closures of cylinder deformations of $M$ gives absolutely no indication that $M$ has large orbit closure.

Problem \ref{p2} is an extreme special case of the following more open ended question: to what extent does the Cylinder Deformation Theorem account for all deformations of a translation surface $M$ which remain in the orbit closure of $M$?

%
%
%


\bibliography{mybib}{}
\bibliographystyle{amsalpha}
\end{document}